\documentclass[12pt,a4paper]{amsart}
 \pdfoutput=1
 
 \usepackage{amsmath,amssymb,amsthm}
 \usepackage{bm}  
 \usepackage{mathtools}
\usepackage{tikz}
\usetikzlibrary{arrows}
 \usepackage{thmtools}
 \usepackage{float}
 
 \usepackage{xcolor} 	
 \usepackage{hyperref}
 \hypersetup{
 	colorlinks,
     linkcolor={red!60!black},
     citecolor={green!60!black},
     urlcolor={blue!60!black},
 }
 \usepackage[maxbibnames=99]{biblatex}
 \usepackage[utf8]{inputenc}
  \usepackage{csquotes}
 \usepackage[T1]{fontenc}
 \usepackage{lmodern}
 \usepackage[babel]{microtype}
 \usepackage[english]{babel}
 
 \usepackage{geometry}
 \usepackage{enumitem}

\theoremstyle{plain}
\newtheorem{theorem}{Theorem}[section]
\newtheorem{fact}[theorem]{Fact}

\newtheorem{claim}[theorem]{Claim}
\newtheorem{corollary}[theorem]{Corollary}
\newtheorem{lemma}[theorem]{Lemma}

\newtheorem{observation}[theorem]{Observation}

\newtheorem{conjecture}[theorem]{Conjecture}

\newtheorem*{theorem*}{Theorem}
\newtheorem*{corollary*}{Corollary}

\theoremstyle{definition}
\newtheorem{remark}[theorem]{Remark}

\title{The Lovász-Cherkassky theorem in infinite graphs}

\author{Attila Jo\'{o}}
\thanks{Funded by the Deutsche Forschungsgemeinschaft (DFG, German
Research Foundation)-513023562 and partially by NKFIH OTKA-129211}
\address{Attila Jo\'{o},
Department of Mathematics, University of Hamburg, Bundesstra{\ss}e 55 (Geomatikum), 20146 Hamburg, Germany}
\email{attila.joo@uni-hamburg.de}


\begin{document}

\begin{abstract}
Infinite generalizations of theorems in finite combinatorics were initiated by Erdős due to his famous Erdős-Menger 
conjecture (now 
known as the Aharoni-Berger theorem) that extends Menger's theorem to infinite graphs in a structural way. We prove a 
generalization of this manner of the classical result about packing edge-disjoint $ T $-paths in an ``inner Eulerian'' setting 
obtained by Lovász and Cherkassky independently in the '70s.  
\end{abstract}
\maketitle

\section{Introduction}

\subsection{Motivation}
Paul Erdős was one of the most influential mathematicians in modern combinatorics. One of his excellent qualities was the 
ability to work equally well with finite and infinite combinatorial structures. He conjectured as a student that Menger's 
theorem remains true in infinite graphs in the following structural sense: For every directed graph $ D$  and $ A, B\subseteq 
V(D) $, there is a family $ \mathcal{P} $ of disjoint directed $ 
AB $-paths and an $ AB $-separator $ S $ consisting of  precisely
one vertex from each path in $ \mathcal{P} $. This conjecture gave a great amount of work for 
several mathematicians (Podewski, Steffens, Aharoni, Berger, Nash-Williams, Shelah, etc.) for decades after it was 
eventually proved in its whole generality \cite{aharoni2009menger}. It also inspired Aharoni and Nash-Williams, who 
themselves also have a significant contribution both to finite and infinite combinatorics, to formulate conjectures in 
infinite combinatorics 
extending their finite counterparts.

For example, Aharoni conjectured that if all the finite induced subgraphs of a graph are perfect and it does not contain an 
infinite independent set of vertices, then it  
admits a vertex-colouring 
where there is no monochromatic edge and there is a clique on which all colours appear. This was verified for 
chordal graphs in \cite{aharoni1995strongly}. Nash-Williams' conjecture about the intersection of infinite matroids (see 
\cite{aharoni1998intersection}) 
generalizes 
structurally the Matroid 
intersection theorem of Edmonds and is an active research direction \cite{aigner2018intersection, 
bowler2021intersection, joo2021MIC}. 
His conjecture about the 
existence of a $ k $-arc-connected orientation of $ 2k $-edge-connected graphs, that he proved for finite graphs, is also 
intensively investigated \cite{thomassen2016orientations, assem2023nash}. 

These examples indicate that (structural) infinite generalizations of classical theorems in finite combinatorics are rich sources  
of deep, new mathematical ideas. The starting point of our investigation is a theorem obtained by Lovász and Cherkassky 
independently in the '70s. We need to introduce a few notions to state it.

\subsection{The main result}
 A grapht $ G $ is a graph together with a fixed subset $ T $ of vertices. A grapht is inner 
Eulerian 
if there is no vertex set $ X $ disjoint from $ T $ that defines an odd cut\footnote{In a finite grapht assuming this only for 
singletons implies that 
this holds for every $ X $.}.  A $ T $-path is a 
path between two distinct elements of $ T $ without inner points in $ T $. For disjoint  $ X,Y\subseteq V(G) $, let 
$ \lambda_G(X, Y) $ denote the 
maximal number of pairwise edge-disjoint paths between $ X$ and $ Y $ in $ G $.

\begin{theorem}[Lovász-Cherkassky theorem, \cite{lovasz1976some, cherkasskiy1977}]\label{t: LCh}
In every finite, inner Eulerian grapht $ G $ the 
maximal number of pairwise edge-disjoint $ T $-paths is equal to
\[ \frac{1}{2}\sum_{t\in T}\lambda_G(t, T-t).  \]

\end{theorem}
If $ \mathcal{P} $ is a system of edge-disjoint  $ T 
$-paths, then $ \left|\mathcal{P}\right| $ can be calculated by counting for each $ t \in T $  how many $ P\in 
\mathcal{P} $ has $ t $ as an 
end-vertex and dividing the sum of these quantities by two. Since $ \lambda_G(t,T-t) $ is clearly an overestimation of the 
number of paths between 
$ t $ and $ T-t $ in $ \mathcal{P} $, the expression in Theorem \ref{t: LCh} is indeed an upper bound for the maximal size of 
a 
family of edge-disjoint $ T $-paths (regardless if $ G $ is inner Eulerian). Having $ \lambda_G(t,T-t) $ paths in $ 
\mathcal{P} 
$ ending at
$ t $   can be rephrased via 
Menger's theorem by the existence of a cut separating $ t $ and $ T-t $ that 
consists of precisely 
one edge from each $ P\in \mathcal{P} $ ending at $ t $. The equality in Theorem \ref{t: LCh} means that such 
a cut  exists for 
each $ t $. This indicates a structural formulation of 
Theorem \ref{t: LCh}. The main result of this paper is that this structural formulation\footnote{The literal, quantitative 
infinite 
generalization also holds but, as in the case of Menger's theorem, it is much weaker than the structural one.} is true 
regardless of the size of $ G $:
\begin{theorem}[Infinite Lovász-Cherkassky theorem]\label{thm: LCh inf}
In every inner Eulerian graft  $G $ there exists a system 
$\mathcal{P}$ of edge-disjoint $T$-paths such that for every $t\in T$ there is a cut separating $ t $ and $ T-t $ that consists 
of precisely 
one edge from each path in $  \mathcal{P} $ that has $ t $ as an end-vertex.
\end{theorem}
\subsection{Difficulties compared to the finite version}
In order to prove Theorem \ref{thm: LCh inf}, we are facing some difficulties that are not arising if only finite 
graphts are considered. Some of the proofs of the finite version start with a hypothetical counterexample for which $ 
\left|V\right|+\left|E\right| $ is 
minimal, reduces this quantity by one in a way, uses the fact that the result is no longer a counterexample and 
shows 
based on that, the original was also not a counterexample. Since reducing an infinite cardinal by one is not meaningful, 
such 
proofs cannot be adapted directly. 

There exists an augmenting path based algorithm for the `packing edge-disjoint $ T 
$-paths' 
problem \cite{iwata2020blossom} that also works in the infinite (not necessarily inner Eulerian) setting.  Even so, in general 
no well-defined 
limit object can be guaranteed after infinitely many iterations of augmenting paths.

A key observation in Lovász' proof is that for each $ t\in T $ all the numbers in   \[ \{ d(X):\ t\in X \subseteq V\setminus 
(T-t) \} \]  have the same parity (namely the parity of $ d(t) $). Indeed,  the sum $ \sum_{x\in X}d(x) $ is equal to $ d(X) $ 
plus twice the number of edges 
spanned by $ X $ and all 
summands except maybe $ d(t)$ are even. This allows Lovász to delete two edges ``safely'' in a certain setting applying 
Menger's theorem and the fact that if $ m<n $ and they have the same parity then $ m+2\leq n $.  Parity-based arguments 
do not work in the infinite case because all the involved 
cardinals might be infinite as well as $ X $ itself.

Despite of 
these hindrances some of the ideas in Lovász' proof can be adapted. The `safe deletion of two edges' by Lovász plays an 
important role in our proof as well (Lemma \ref{lem: deletion of two edges}) but the parity-based argument needs to be 
replaced 
by a more complex structural one. 
The `splitting off' technique of Lovász also appears in our proof. It does not 
reduce $ \left|E\right| $ in general but it reduces a finite quantity associated with a hypothetical counterexample for Lemma 
\ref{lem: 1path make}. To handle uncountable graphs,  the key is (beyond elementary submodels and singular 
compactness) Lemma \ref{lem: no stat error} which 
was inspired by the works of Aharoni, Nash-Williams and Shelah on the existence of transversals \cite{aharoni1983general, 
aharoni1984another}. Roughly 
speaking, the abstraction of their idea says that if we have an uncountable 
structure with a certain property and consider 
an increasing, continuous sequence of substructures whose length is an uncountable regular cardinal, then (in the right 
setting) 
there is a closed, unbounded subset of terms each of whose removal preserves the condition in question. The exact 
implementation of this abstract idea varies from problem to problem (see \cite{aharoni1984konig, aharoni1988matchings, 
aharoni2009menger}).  
In 
our case we can maintain in 
Lemma \ref{lem: no stat error} only a weakening of the property that we initially assumed. Even so, being aware of their 
technique was essential in guiding our proof efforts in the right direction.

\subsection{The structure of the paper}
In Section \ref{sec: notation} we introduce the notation. The former results we need for the proof are summarized in Section 
\ref{sec: premin}. The proof of the main result is divided into three sections. 

In the first one 
(Section \ref{sec: redu linkage}), we discuss a simple method to ``lift up'' systems of edge-disjoint $ T $-paths to the original 
grapht that are obtained after certain contractions.  In particular, we reduce 
the main result to 
a theorem (Linkage theorem \ref{thm: linkage}) in which we want to cover every edge incident with any $ t\in T $ by a 
system of edge-disjoint $ T $-paths under the condition that for any single $ t\in T $, the edges incident with $ t $ can be 
covered by such a system. The 
proofs of this 
section are relatively simple. They are based on the iterated application of 
the infinite Menger's theorem \cite{aharoni2009menger} and the infinite version of Pym's theorem  \cite{diestel2006cantor}. 

Section \ref{sec: reduc countable} is dedicated to the reduction of the Linkage theorem to its 
special case where the grapht is 
countable. Only this section requires some basic knowledge about combinatorial set theory (stationary sets and elementary 
submodels).
The already mentioned Lemma \ref{lem: no stat error} plays a key role by guaranteeing that we can maintain 
the premise of the Linkage theorem ``most of the time'' while cutting up an uncountable grapht into smaller pieces.

In Section \ref{sec: countable case}, we prove the countable case of the Linkage theorem. This is a 
somewhat simplified version of our previous work \cite{joo2023lovcher} in which we solved the countable case of the 
problem. We decided to include this for 
the sake of completeness but moved it to the end, thus readers already familiar with the proof of the countable case 
can skip it.

In the last section (Section \ref{sec: outlook}), we formulate a ``meta-conjecture'' together with a couple of specific 
conjectures 
concerning infinite generalizations of further classical results.
\section{Notation}\label{sec: notation}
\subsection{Graph theory}
In graphs, we allow parallel edges but not loops. Technically we represent a graph as a triple $ \boldsymbol{G=(V,E,I)} $ 
where 
the \emph{incidence function} $ \boldsymbol{I}:E\rightarrow [V]^{2} $ determines the end-vertices of the edges. An 
edge-partition of a graph $ 
G $ is a family of subgraphs whose edge sets partition $ E(G) $. 
For $ X\subseteq V $ let $ \boldsymbol{\delta_G(X)}:=\{ e\in E:  \left|I(e)\cap X\right|=1  \} $ and we write $ 
\boldsymbol{d_G(X)} $ for $ 
\left|\delta_G(X)\right| $.  If  a grapht $ G $ is obvious from the context, then we omit the subscript, furthermore, for a 
singleton $ 
\{ 
x \} $ we write simply $ \delta(x) $ and $ d(x) $.   All the paths in the 
paper are finite.  For  disjoint $ 
A,B\subseteq V $ an $ \boldsymbol{AB} $\textbf{-path}
is a path with one endpoint 
in $ A $ 
the other in $ B $ and no internal vertices in $ A\cup B $. Sometimes we need to combine paths to obtain new ones. We 
illustrate our notation by an example: Let $ P $ and $ Q $ be paths with $ s\in V(P) $, $ t\in V(Q) $ and $ e\in E(P)\cap E(Q) 
$. Then $ \boldsymbol{sPeQt} $ is the graph that consists of the smallest subpath of $ P $ containing $ s $ and $ e $ and the 
smallest subpath of $ Q $ containing $ t $ and $ e $. We use this only when the resulting graph is also a path. We call  $ 
C\subseteq E $ a \emph{cut} 
if $ C=\delta(X) $ for 
some $ X\subseteq V $. It is an $ \boldsymbol{AB} $\textbf{-cut} if 
$ G-C $ does not 
contain any $ AB $-path. We say that the edge set $ C $ and   the system $ \mathcal{P} $ of edge-disjoint paths are 
\emph{orthogonal} to each other if $ C 
$ consists of one edge from each $ P\in \mathcal{P} $. An \emph{Erdős-Menger} $ AB $\emph{-cut} is an $ AB $-cut $ C $ 
for which there 
exists a system $ \mathcal{P} $ of edge-disjoint $ AB $-paths such that $ C $ is orthogonal to $ \mathcal{P} $. The $ 
\boldsymbol{A} 
$\textbf{-side} of an $ AB $-cut $ C $ is the set of vertices reachable from $ A $ without using edges from $ C $.
For a $ U\subseteq V $ and a family $ \mathcal{F}=\{ X_u: u\in U \} $ of pairwise disjoint 
subsets of $ V $ with $ X_u\cap U=\{ u \} $, we define the graph  $ \boldsymbol{G/\mathcal{F}} $  obtained 
from $ G $ by \emph{contracting} $ X_u$ to $u $ for $ u\in U $ and deleting the resulting loops. Formally,  
$V(G/\mathcal{F}):= 
(V\setminus \bigcup \mathcal{F} 
)\cup U,\ E(G/\mathcal{F}):= E\setminus \{ e\in E: (\exists u\in U) I(e)\subseteq X_u \} $ and  $ 
I(G/\mathcal{F})(e):=\{ i_\mathcal{F}(u), i_\mathcal{F}(v) \} $  where  $ I(e)=\{ u,v \} $ and

\[ i_\mathcal{F}(v) =\begin{cases} v &\mbox{if } v\notin \bigcup \mathcal{F} \\
u & \mbox{if } u\in X_u .  
\end{cases} \]

 A quadruple $ \boldsymbol{G=(V,E,I, T)} $ is a \emph{grapht} if $ (V,E,I) $ is a graph and $ T\subseteq V $ with $ 
 \left|T\right|\geq 2 $. Definitions 
 corresponding to 
 graphs are extended to graphts in the natural way.  We refer to the elements of 
 $ T $ as \emph{terminal points}.  A grapht is inner 
 Eulerian 
 if there is no $ X\subseteq V\setminus T $ where $ d(X) $ is an odd natural number. A $ T $\emph{-path} is a 
 path between two distinct terminal points without inner points in $ T $. For $ t\in T $, we write $ t $-cut as an abbreviation 
 of $ t(T-t) $-cut. A system $ \mathcal{P} $ of edge-disjoint $ T $-paths \emph{links} $ t $ if it covers $ \delta(t) $.  It is a 
 \emph{linkage} for $ t\in T $ if links $ t $ and minimal with respect to this property (i.e. each $ P\in \mathcal{P} $ ends in $ 
 t $). We say that $ t $ is linked if there is a $ \mathcal{P} $ that links it. Finally,  
 $ 
 \mathcal{P} $ is a \emph{perfect linkage} of $ G $ if it links every $ t\in T $. A grapht 
 satisfies the \emph{linkability condition} if every $ t\in T $ is linked in it. For a set $ M $ and grapht 
 $ G $, we 
 abuse the 
 notation and let    
  \begin{align*}
  \boldsymbol{G \cap M}&:=(V,E\cap M,I \!\!\upharpoonright\!\! (E\cap M) ,T)\\
      \boldsymbol{G \setminus M}&:= (V,E\setminus M,I\!\!\upharpoonright\!\! (E\setminus M), T).
  \end{align*}
  
\subsection{Set theory}
We use standard set-theoretic notation. The variables $ \alpha,\beta $ and $ \gamma $ are standing for ordinal 
numbers, while 
$ \kappa $ denotes
cardinals. The smallest 
infinite cardinal, i.e. the set of the natural numbers is denoted by $ \boldsymbol{\omega} $.   The cofinality of $\kappa $ is  $ 
\boldsymbol{\mathsf{cf}(\kappa)} $. Let $ \kappa $ be 
 an infinite cardinal. A sequence $ \left\langle M_\alpha:\ \alpha<\kappa  \right\rangle $ of sets is increasing ($ \in 
 $-increasing) if $ M_\beta 
 \subseteq 
 M_\alpha $ ($ M_\beta \in
  M_\alpha $) for every $ \beta<\alpha<\kappa $. An increasing sequence is continuous if $ 
  M_\alpha=\bigcup_{\beta<\alpha}M_\beta $ for each limit ordinal $ 
  \alpha<\kappa $. Suppose that $ \kappa=\mathsf{cf}(\kappa)>\omega $. A set $ C\subseteq \kappa $ is a 
  \emph{club} of $ \kappa $ if it 
  is unbounded in $ 
 \kappa $ and closed with respect to 
 the order 
 topology (i.e. $ \sup B:=\bigcup B \in C $ for every $ B\subseteq C $ bounded in $ \kappa $). A set $ S\subseteq \kappa $ is 
 \emph{stationary} in $ \kappa $ if it meets every club of $ \kappa $.

 \subsection{Elementary submodels}
 Let $ \varphi $ be a formula in the first-order language of set theory with free variables $ 
 v_1,\dots, v_n $. For a set $ M $, the formula $ 
 \varphi^{M} $ is obtained from $ \varphi $ by  the relativization of the quantifiers to $ M $ (i.e. $ \forall v(\dots) $ is 
 replaced by 
 $ {\forall v( v\in M \Longrightarrow(\dots))}$ and 
 $\exists v(\dots)   $ by $ \exists v( v\in M \wedge (\dots)) $).  The set $ M $ is a
  $ \varphi 
 $-elementary submodel of the universe\footnote{It is also common to approximate the universe by a ``big'' but set-sized 
 structure first and take elementary submodels of that. From the perspective of applications, it does not make a difference.} if 
 for 
 every $ x_1,\dots, x_n\in M $ we have $  
 \varphi(x_1,\dots, 
 x_n)\Longleftrightarrow 
 \varphi^{M}(x_1,\dots, x_n) $. 
 Let $ \Sigma $ be a finite set of formulas. Then we say that $ M $ is a $ \Sigma $-elementary submodel of the universe if $ 
 M $ is $ \varphi 
 $-elementary for each $ 
 \varphi \in \Sigma $. 
 
 Let a finite set $ \Sigma $ of formulas, that contains everything relevant for our proofs, be fixed 
 through the paper. To improve the flow of words, we say simply ``elementary submodel'' instead of ``$ \Sigma $-elementary 
 submodel of the universe''. For more details concerning elementary submodels, we refer to 
 \cite{soukup2011elementary}.
\section{Preliminaries}\label{sec: premin}
\begin{lemma}[{Fodor's lemma, \cite[Theorem 8.7]{jech2002set}}]\label{lem: Fodor}
 If $ \kappa=\mathsf{cf}(\kappa)>\omega $, $ S\subseteq \kappa $ is stationary and $ f: S\rightarrow \kappa $ is a regressive 
 function 
 (i.e. $ f(\alpha)<\alpha $ for every $ \alpha \in S\setminus \{ 0 \} $), then there is a stationary $ S'\subseteq S $ and a $ 
 \gamma<\kappa $ such that $ f(\alpha)=\gamma $ for each $ \alpha\in S' $.
 \end{lemma}
\begin{theorem}[{Nash-Williams, \cite[p. 235 Theorem 3]{nash1960decomposition}}]\label{thm: cycle decomp NW}
A graph $ G $ can be edge-partitioned into cycles if and only if there is no cut $ C $ in $ G $ such that $ \left|C\right| $ is an 
odd natural 
number.
\end{theorem}
\begin{corollary}\label{cor: inner Eulerian grapht partionon}
A grapht can be edge-partitioned into cycles and $ T $-paths if and only if it is inner Eulerian.
\end{corollary}
\begin{proof}
On the one hand, being inner Eulerian is clearly necessary because for $ X\subseteq V(G)\setminus T $, the contribution of 
a $ T $-path or a cycle to $ d_G(X) $ is even. On the other hand, $ G $ being inner Eulerian means that the graph $ G/T $ we 
obtain from $ G $ by contracting $ T $ into a single point has no odd cuts. Thus by Theorem \ref{thm: cycle decomp NW}, 
 $ G/T $ admits an edge-partition into cycles which provides an edge-partition of $ G $ into cycles and $ T $-paths.
\end{proof}

Menger's theorem and the other connectivity-related results that we recall have four versions depending on 
whether the graph is directed and if we consider vertex-disjoint or edge-disjoint paths. We state always just the undirected 
edge-variant which can be derived from any other versions.

Let a grapht $ G=(V,E, I, \{ s,t \}) $  be fixed for the rest of the section.

\begin{theorem}[Aharoni and Berger, \cite{aharoni2009menger}]\label{t: Inf Menger}
There is a system $ \mathcal{P} $ of edge-disjoint $ st $-paths and an $ st $-cut  $ C $ which is \emph{orthogonal} to $ 
\mathcal{P} $.
\end{theorem}

\begin{theorem}[Diestel and Thomassen, \cite{diestel2006cantor}]\label{t: Pym}
Let $ \mathcal{P} $ and  $ \mathcal{Q} $ be systems of edge-disjoint $ st $-paths. Then there exists a system $ \mathcal{R} 
$ of edge-disjoint $ st 
$-paths such that $ \delta_\mathcal{R}(s) \supseteq  \delta_\mathcal{P}(s) $ and $ \delta_\mathcal{R}(t) \supseteq  
\delta_\mathcal{Q}(t) $.
\end{theorem}

For Erdős-Menger $ st $-cuts $ C $ and 
$ C' $, we write $ C \trianglelefteq C' $, if for every $ st$-path $ P $ if $ <_P $ is the linear order on $ E(P) $ corresponding 
to 
the  $ s \rightarrow t $ direction of $ P $, then $ \min_{<_P} (C \cap E(P) )\leq \min_{<_P} (C'\cap E(P)) $. These definitions 
extend naturally for the setting where we have disjoint, nonempty sets $ S $ and $ T $ instead of distinct points $ s $ 
and $ t $ by contracting $ S $ and $ T $ to one of their respective elements.

\begin{theorem}[\cite{joo2019complete}]
The Erdős-Menger $ st $-cuts form a complete lattice with respect to $ \trianglelefteq  $. 
\end{theorem}

\begin{lemma}[Augmenting paths for Menger's theorem, \cite{bohme2001menger}]\label{lem: aug path}
Let $ \mathcal{P} $ be a system of edge-disjoint $ st $-paths. Then there exists  
\begin{enumerate}
\item\label{item: no aug} either an $ st $-cut $ C $ orthogonal to 
$ \mathcal{P} $,
\item\label{item: aug}  or a system $ \mathcal{Q} $ of 
edge-disjoint $ st $-paths for which $ \mathcal{P}\triangle \mathcal{Q} $ is finite,
$ \delta_\mathcal{Q}(s) \supset   \delta_\mathcal{P}(s) $ with $ \left|\delta_\mathcal{Q}(s) \setminus   \delta_\mathcal{P}(s) 
\right|=1$   and $ \delta_\mathcal{Q}(t) 
\supset  \delta_\mathcal{P}(t) $ with $ \left|\delta_\mathcal{Q}(t) \setminus   \delta_\mathcal{P}(t) 
\right|=1$.
\end{enumerate}
\end{lemma}

For the following couple of lemmas, we could not find really fitting references, thus we include their short proofs. As far we 
know, the 
phenomenon of  ``tight'' sets in the context of the infinite version of König's theorem \cite{aharoni1984konig} and beyond was
discovered by Podewski and Steffens \cite{podewski1976injective, steffens1977matchings} and it is known that it holds in 
the 
context of the infinite Menger's theorem as well: 

\begin{lemma}\label{lem: tight cut}
Suppose that $ s $ is linked and every path-system that links $ s $ uses the 
edge $ e\in E\setminus 
\delta(s) $. Then the $ \trianglelefteq $-smallest Erdős-Menger $ st $-cut $ C $ that contains $ e $ is \emph{tight} in the 
sense 
that 
every linkage of $ s $ is orthogonal to $ C $.
\end{lemma}
\begin{proof}
Suppose for a contradiction that $ G $ and $ e $ form a counterexample. We can assume without loss of generality that $ 
C=\delta_G(t) $ because after the contraction of the $ t $-side of $ C $ to $ t $, it remains a counterexample. By the indirect 
assumption there exists a linkage $ \mathcal{P} $ for $ s $ that
does not cover $ \delta_G(t) $.  Let $ P_e $ be the unique path in $ \mathcal{P} $ through $ e $. By applying Lemma 
\ref{lem: aug path} with the path-system  $ \mathcal{P}-P_e $ in $ G-e $, the alternative (\ref{item: aug}) is impossible 
because the resulting 
path-system $ 
\mathcal{Q} $ would 
cover $ \delta_{G'}(s)=\delta_G(s) $ without using $ e $. It follows that there is an $ st $-cut $ C' $ in $ G-e $ orthogonal to $ 
\mathcal{P}-P_e 
$. The path $ P_e-e $ lives in $ G-e $ and does not meet $ C' $, thus $ C'+e $ is an Erdős-Menger $ st $-cut in $ G $ with the 
same sides as 
$ C' $. Clearly $ C'+e \trianglelefteq \delta_G(t) $. Furthermore, $ C'+e \neq \delta_G(t) $ because $ C'+e $ is orthogonal to $ 
\mathcal{P} $ and $ \delta(t) $ is not. This contradicts the $ \trianglelefteq $-minimality of $ C $.
\end{proof}

The following lemma is about a fundamental property of the so-called (infinite) gammoids 
\cite{carmesin2014topological, 
afzali2015infinite}:
\begin{lemma}\label{lem: missing exactly one}
Suppose that $ \mathcal{P} $ is a linkage of $ s $ with $ \left|\delta(t)\setminus E(\mathcal{P}) 
\right|=1$ and there is no linkage $ \mathcal{P}' $ for $ s $ with
$ \delta_{\mathcal{P}'}(t)\subsetneq \delta_{\mathcal{P}}(t) $. Then $ \left|\delta(t)\setminus E(\mathcal{Q}) \right|\leq1$ 
holds for every linkage $ \mathcal{Q} $ for $ s $.
\end{lemma}
\begin{proof}
Let $ e $ be the unique element of $ \delta(t)\setminus 
E(\mathcal{P}) $. Suppose for a contradiction that there exists a linkage $ \mathcal{Q} $ for $ s $ with $ 
\left|\delta(t)\setminus E(\mathcal{Q}) 
\right|>1$. Then we must have $ e\in E(\mathcal{Q}) $, since otherwise $ \mathcal{P}'=\mathcal{Q} $ contradicts 
the premise of the lemma. Let $ Q_e $ be the unique path in $ \mathcal{Q} $ through $ e $ and let $ 
\mathcal{Q}':=\mathcal{Q}-Q_e $. By applying Lemma \ref{lem: aug path} with  
the path-system  $ \mathcal{Q}' $ in $ G-e $ the alternative (\ref{item: aug}) is impossible because the resulting path-system 
$ \mathcal{P}'$ would 
contradict
the premise of the lemma. It follows that there is an $ st $-cut $ C $ in $ G-e $ orthogonal to $ \mathcal{Q}' $. But then the 
path-system $ \mathcal{P}' $ that we obtain by taking each  $ P\in \mathcal{P} $ from $ s $ until the first common 
edge $ e_P $ with $ C $ and continue up to $ t $ along the unique path in $ \mathcal{Q}' $ through $ e_P $ contradicts again 
the premise of the lemma.
\end{proof}

\section{Reduction to the Linkage theorem}\label{sec: redu linkage}
In order to cover $ \bigcup_{t\in T}\delta_{G}(t) $ by edge-disjoint $ T $-paths, i.e. to have a perfect linkage, it is obviously 
necessary that for each 
individual $ t\in T $ the edge set $ \delta_G(t) $ can be covered by such paths, i.e. that the linkability condition holds.  
Maybe surprisingly, in the inner 
Eulerian case 
this is already sufficient as well:
\begin{theorem}[Linkage theorem]\label{thm: linkage}
Every inner Eulerian graft  that satisfies the linkability condition admits a perfect linkage.
\end{theorem}

In the rest of the section, we show that Theorem \ref{thm: linkage} implies our main result \ref{thm: LCh inf} and introduce 
some 
tools that we need later.

Let a grapht  $ G=(V,E,I,T) $ be fixed for the rest of the section. 
\begin{observation}[Lifting lemma]\label{obs: lift up}
Let $ \mathcal{F}=\{ X_t:\ t\in T \} $ be a family of pairwise disjoint subsets of $ V $ where  $ t\in X_t $ and $ X_t $ is 
\emph{boundary-linked} in the sense that there is a 
system $ \mathcal{P}_{t}=\{ P_{t,e}:\ e\in \delta_G(X_t) \} $ of edge-disjoint paths where $ e $ 
is 
the first edge and $ t $ is the last vertex of $ P_{t,e} $ (in suitable orientation). 

Then for every  system $ \mathcal{P} $ of edge-disjoint $ T $-paths in $ 
G/\mathcal{F}$, there is a system $ \mathcal{P}' $ of edge-disjoint $ T $-paths in $ G $ such that for every $ P\in 
\mathcal{P} $ there is a $ P'\in \mathcal{P}' $ has the same endpoints as $ P $ and $ E(P')\supseteq E(P) $. Furthermore, this 
$ \mathcal{P}' $ can be chosen 
in such a 
way that it links $ t $ for every  $ t\in T $ for which $ \mathcal{P} $ links $ t $ 
and $ \mathcal{P}_t $ covers $ \delta_G(t) $.
\end{observation}
\begin{proof}
For a $ P\in \mathcal{P} $ we obtain $ P' $ by uniting $ P $ and the paths $ P_{t_i, e_{i}} $ for $ i\in \{ 0,1 \} $ where $ t_0 
$ and $ t_1 $ are the endpoints of $ P $ and $ e_i $ is the unique edge of $ P $ incident with $ t_i $ in $ G/\mathcal{F} $ 
(see Figure \ref{fig: lif up}). 

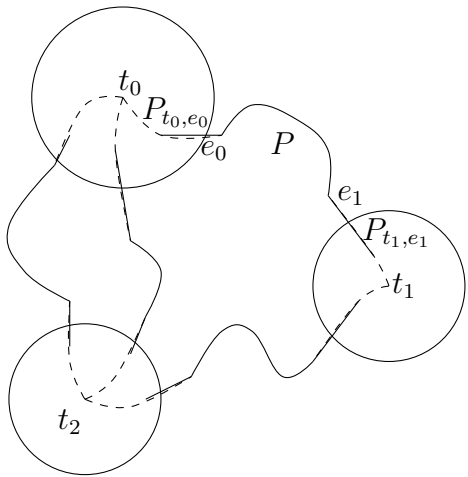
\begin{figure}[h]
\centering

\begin{tikzpicture}

%boundary linkages
\draw  (-2,0.5) node (v1) {} ellipse (1 and 1);
\draw  (-1.5,4.5) node (v2) {} ellipse (1.2 and 1.2);
\draw  (2,2) node (v3) {} ellipse (1 and 1);
\draw[dashed]  plot[smooth, tension=.7] coordinates {(v1) (-2.2,1) (-2.2,1.8)};
\draw[dashed]  plot[smooth, tension=.7] coordinates {(v1) (-1.6,0.8) (-1.2,1.6)};
\draw[dashed]  plot[smooth, tension=.7] coordinates {(v1) (-1.4,0.4) (-0.6,0.8)};
\draw[dashed]  plot[smooth, tension=.7] coordinates {(v2) (-2,4.4) (-2.4,3.6)};
\draw[dashed]  plot[smooth, tension=.7] coordinates {(v2) (-1.6,3.8) (-1.4,2.6)};
\draw[dashed]  plot[smooth, tension=.7] coordinates {(v2) (-1,4) (-0.2,4)};
\draw[dashed]  plot[smooth, tension=.7] coordinates {(v3) (1.8,2.4) (1.2,3.2)};
\draw[dashed] plot[smooth, tension=.7] coordinates {(v3) (1.6,1.8) (1,1)};

%T-paths
\draw  plot[smooth, tension=.7] coordinates {(-2.2,1.8) (-2.8,2.2) (-3,2.8) (-2.4,3.6)};
\draw  plot[smooth, tension=.7] coordinates {(-1.2,1.6) (-1,2.2) (-1.4,2.6)};
\draw  plot[smooth, tension=.7] coordinates {(-0.6,0.8) (-0.2,1.4) (0.2,1.4) (0.6,0.8) (1,1)};
\draw  plot[smooth, tension=.7] coordinates {(1.2,3.2) (1.2,3.8) (0.8,4.2) (0.2,4.4) (-0.2,4)};

%names

\node at (-1.4,4.7) {$t_0$};
\node at (2.2,2) {$t_1$};

\node at (-2.2,0.2) {$t_2$};
\node at (1.5,3.2) {$e_1$};
\node at (-0.3,3.8) {$e_0$};
\node at (0.6,3.9) {$P$};
\node at (2.1,2.7) {$P_{t_1, e_1}$};
\node at (-0.8,4.3) {$P_{t_0, e_0}$};

\draw  plot[smooth, tension=.7] coordinates {(-2.2,1.8) (-2.2,1.2)};
\draw  plot[smooth, tension=.7] coordinates {(-1.2,1.6) (-1.4,1.1)};
\draw  plot[smooth, tension=.7] coordinates {(-0.6,0.8) (-1.2,0.5)};

\draw  plot[smooth, tension=.7] coordinates {(-2.4,3.6) (-2.2,4)};

\draw  plot[smooth, tension=.7] coordinates {(-1.4,2.6) (-1.6,3.8)};
\draw  plot[smooth, tension=.7] coordinates {(-0.2,4) (-1,4)};
\draw  plot[smooth, tension=.7] coordinates {(1.2,3.2) (1.8,2.4)};
\draw  plot[smooth, tension=.7] coordinates {(1,1)};
\draw  plot[smooth, tension=.7] coordinates {(1,1) (1.6,1.8)};
\end{tikzpicture}
\caption{The Lifting lemma (Observation \ref{obs: lift up})} \label{fig: lif up}
\end{figure}
\end{proof}

\begin{lemma}\label{lem: joker}
For every $ T'\subseteq T$ there exists a family $ \mathcal{F}=\{ X_t:\ t\in T \} $ of disjoint 
vertex sets with $t\in  X_t $  such that $ X_t=\{ t \} $ for $ t\in T\setminus T' $ and for each $ s\in T' $: $ \delta(X_s) $ is 
the $ \trianglelefteq $-largest 
Erdős-Menger  
$ X_s (\bigcup_{t \in T-s}X_t)$-cut. Every such $ \mathcal{F} $ has the 
following properties:
\begin{enumerate}[label=(\roman*)]
\item\label{item: joker links} For each $ t\in T $ for which $ t $ is linked in $ G $ there is a 
system $ \mathcal{P}_{t}=\{ P_{t,e}:\ e\in \delta_G(X_t) \} $ of edge-disjoint paths covering $ \delta_G(t) $ where $ e $ 
is 
the first edge and $ t $ is the last vertex of $ P_{t,e} $ (in suitable orientation);
\item\label{item: joker EM} For every $ t \in T' $: $ \delta_{G/\mathcal{F}}(t) $ is the unique Erdős-Menger $ t$-cut in $ 
G/\mathcal{F} $;
\item\label{item: joker solution}  If $ G/\mathcal{F} $ has a perfect linkage then $ G $ has a system $ \mathcal{P} $ of $ T 
$-paths with the properties described in Theorem 
\ref{thm: LCh inf}. 
\end{enumerate}
\end{lemma}
\begin{proof}
To show the existence, let $ T'=\{ t_\alpha:\ \alpha<\kappa \} $. If $ X_\beta $  is already defined for $ \beta<\alpha $, then 
we contract $ X_\beta $ 
to $ t_\beta $ for each $ \beta<\alpha $ and define $ X_\alpha $ as the $ t_\alpha $-side of the $ \trianglelefteq $-largest 
Erdős-Menger  $ t_\alpha $-cut of the resulting grapht.

To show \ref{item: joker links}, observe that whenever $ \delta_G(X_t) $ is an Erdős-Menger $ t $-cut with $ t\in X_t $, 
then $ X_t $ is boundary-linked (in the sense of Observation \ref{obs: lift up}).  If in addition $ t $ is 
linked in $ G $, then we claim that
Theorem \ref{t: Pym} allows us to choose the path-system $ 
\mathcal{P}_{t} $ witnessing being boundary-linked in such a way that it covers $ \delta_G(t) $. Indeed, contract $ 
V\setminus X_t $ to an $ s\in T\setminus V 
$. Then in the resulting grapht $ G' $ both $ s $ and $ t $ are linked, thus by Theorem \ref{t: Pym} one can link both by a 
single path-system.

Property \ref{item: joker EM} holds because any further Erdős-Menger $ t$-cut in $ G/\mathcal{F} $  would be an 
Erdős-Menger  $ X_t (\bigcup_{t' \in T-t}X_{t'})$-cut  in $ G $ that is strictly $ \trianglelefteq $-larger than the largest one. 
Finally, \ref{item: joker solution} follows by lifting up a perfect linkage of $ 
G/\mathcal{F} $ to $ G $ in the 
sense of Lemma \ref{obs: lift up} because then the cuts $ \delta_G(X_t) $ for $ t\in T $ are as demanded by Theorem 
\ref{thm: LCh inf}.
\end{proof}
\begin{corollary}
Theorem \ref{thm: LCh inf} is implied by Theorem \ref{thm: linkage}.
\end{corollary}
\begin{proof}
It follows directly from \ref{item: joker solution} of Lemma \ref{lem: joker}.
\end{proof}
\begin{remark}
The proof also shows that the family of cuts in Theorem \ref{thm: LCh inf} can be chosen to be nested.
\end{remark}

\section{Reduction of the Linkage theorem to countable graphts}\label{sec: reduc countable}
In this section, we reduce the Linkage theorem \ref{thm: linkage} to its special case where the edge set of the grapht is 
countable.

\subsection{Basic elementarity arguments}
\begin{observation}\label{obs: prev inner Euler}
If $ M  $ is an elementary submodel and $ G\in M $ is an inner Eulerian grapht, then the grafts $ G \cap M $ and $ G 
\setminus M $ are also inner 
Eulerian. 
\end{observation}
\begin{proof}
By Corollary \ref{cor: inner Eulerian grapht partionon} we can fix an edge-partition $ \mathcal{A}\in M $ of $ G $ into  
cycles and $ T $-paths. Then if $ e \in  E\cap M $, 
then the unique $ H_e\in 
\mathcal{A} $ that contains $ e $ is definable from the parameters $ \mathcal{A} $ and $ e $, thus $ H_e\in M $. Since $ 
H_e $ 
is finite, this implies $ 
E(H_e)\subseteq M 
$. This shows that there is no $ H \in \mathcal{A} $ for which both $ E(H) \cap M $ and $ E(H) \setminus M $ are 
nonempty. But then $ 
\mathcal{A}\cap M $ and $ \mathcal{A}\setminus M $ witness respectively that $ G \cap M $ and $ G \setminus M $ are 
inner Eulerian.
\end{proof}

\begin{observation}\label{obs: in-prev linkability}
If $ M  $ is an elementary submodel and $ G\in M $ is a grapht satisfying the linkability condition, then the graft $ G \cap M 
$ also satisfies the 
linkability condition.  
\end{observation}
\begin{proof}
For $ t\in T \setminus M $ we have $ \delta_{G\cap M}(t)=\emptyset $, thus $ t $ is linked in $ G\cap M $. For  $ t\in T\cap 
M $ let  $ 
\mathcal{P}_t \in M $ be a path-system that links $ t $ in $ G $. 
For 
each $ e\in \delta_G(t)\cap M=\delta_{G \cap M}(t) $, 
the unique $ P_e \in \mathcal{P}_t $ through $ e $ is definable from the parameters $ \mathcal{P}_t $ and $ e $, therefore $ 
P_e\in M $ and hence $ E(P_e)\subseteq M $. Thus $ \mathcal{P}_t\cap M $ links 
$ t $ in $ G\cap M $.
\end{proof}

\begin{observation}\label{obs: out-prev linkability partially}
If $ M  $ is an elementary submodel and $ G\in M $ is a grapht satisfying the linkability condition, then  every $ t\in T \cap 
M $ is linked in $ G\setminus M $. 
\end{observation}
\begin{proof}
Let $ t \in T \cap M $ be fixed and take a system $ \mathcal{P}_t \in M $ that links $ t $. 
We 
claim that for $ e 
\in 
\delta(t)\setminus M $ the unique $ P_e\in \mathcal{P}_t $ through $ e $ cannot have any edges in $ M $. Indeed,  if $ f\in 
E(P_e)\cap M $, then $ 
P_e $ is 
definable from the parameters $ f $ and $ \mathcal{P}_t $, thus $ P_e\in M $ but then $e\in E(P_e)\subseteq M $ which 
contradicts the choice of $ 
e $. 
Therefore the path-system $ \mathcal{P}_t\setminus M $ links $ t $ in $ G\setminus 
M $.
\end{proof}

The linkability of a $ t\in T \setminus M $ in $ G\setminus M $ will take more effort, we handle it 
by a 
``workaround'' in 
the following subsection.

\subsection{Stationarily many violations cannot occur}
\begin{lemma}\label{lem: no stat error}
Suppose that $ \kappa  $ is an uncountable regular cardinal,  $ \left\langle M_\alpha:\ \alpha<\kappa  \right\rangle $ is an 
increasing, 
continuous sequence of elementary submodels with $ \left\langle M_\beta:\ \beta\leq\alpha  \right\rangle\in M_{\alpha+1} $ 
and $ \left|M_\alpha\right|<\kappa $ for each $ \alpha<\kappa $ and $ G=(V,E,I,T)\in M_0 $ is a grapht such that $ \delta(t) $ 
is the unique Erdős-Menger $ t $-cut for $ t\in T $. Then  
there is a club $ U $ of $ \kappa $ such that for  $ \alpha \in U $ the grapht $ G\setminus M_\alpha $ satisfies the linkability 
condition.
\end{lemma}
\begin{proof}
Let $ S\subseteq \kappa $ be the set of those ordinals $ \alpha $ for which the linkability condition 
fails in $ G\setminus M_\alpha $. Suppose for a 
contradiction that $ S $ is stationary. Let $ \left\langle \alpha_\xi:\ \xi<\kappa  \right\rangle $ be 
the increasing enumeration of $ S $.  Fix a well-order $ \prec $ of $ T $ that is in $ M_0 $. We define $ X_\xi\subseteq V $ 
and 
$ t_\xi\in T $ for $ 
\xi<\kappa $ by 
transfinite recursion such that (see Figure \ref{fig: key lemma}):
\begin{enumerate}
\item\label{item: disjoint} the sets $ X_\xi $ are pairwise disjoint,
\item\label{item: t} $ X_\xi \cap T=\{ t_\xi \} $,
\item\label{item: path-system} $ \delta_{G \setminus M_{\alpha_\xi}}(X_\xi) $ is an Erdős-Menger $ t_\xi$-cut in $ G 
\setminus 
M_{\alpha_\xi} $,
\item\label{item: uses an edge down} $ \delta_{G}(X_\xi)\cap M_{\alpha_\xi}\neq \emptyset $,
\item\label{item: definable} $ X_\xi $ and $ t_\xi $  are definable from the parameters $G,\  \left\langle M_\alpha:\ \alpha 
\leq\alpha_\xi  
\right\rangle $ and $ \prec $.
\end{enumerate}

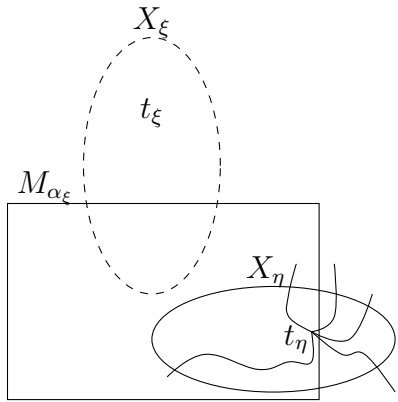
\begin{figure}[h]
\centering

\begin{tikzpicture}

\draw  (-1.9,0.6) rectangle (2.2,-2);
\draw[dashed]  (0,1.1) ellipse (0.9 and 1.7);
\draw  (3.2,-1.9) node (v1) {} ellipse (0 and 0);

%linkages
\draw  plot[smooth, tension=.7] coordinates {(2.1,-1.1) (1.8,-0.9) (1.8,-0.6) (1.9,-0.2)};
\draw  plot[smooth, tension=.7] coordinates {(2.1,-1.1) (2.1,-1.5) (1.8,-1.5) (1.4,-1.6) (0.7,-1.4) (0.2,-1.7)};
\draw  plot[smooth, tension=.7] coordinates {(2.1,-1.1) (2.5,-1.4) (2.8,-1.4) (v1)};
\draw  plot[smooth, tension=.7] coordinates {(2.1,-1.1) (2.4,-0.9) (2.4,-0.2)};
\draw  plot[smooth, tension=.7] coordinates {(2.1,-1.1) (2.6,-1.2) (2.9,-0.6)};
%names
\node at (0,1.8) {$t_\xi$};
\node at (-1.4,0.8) {$M_{\alpha_\xi}$};
\draw  (1.6,-1.2) ellipse (1.6 and 0.7);
\node at (1.9,-1.2) {$t_\eta$};
\node at (0,3) {$X_\xi$};
\node at (1.5,-0.3) {$X_\eta$};
\end{tikzpicture}
\caption{The choice of $ t_\xi $ and $ X_\xi $.} \label{fig: key lemma}
\end{figure}

Let $ \xi<\kappa $ and suppose that $ t_\eta$ and $\ X_\eta $  are defined for $ \eta<\xi $ such that the conditions 
(\ref{item: disjoint})-(\ref{item: definable}) are obeyed so far.  The linkability condition fails in $ 
G\setminus 
M_{\alpha_\xi} $ by assumption. Let $ t_\xi\in T $ be the $ \prec $-smallest witness for this. Note that we must have $ t_\xi 
\in T\setminus 
M_{\alpha_\xi} $ by 
Observation \ref{obs: out-prev linkability partially} and hence $ \delta_{G\setminus M_{\alpha_\xi}}(t_\xi) 
=\delta_{G}(t_\xi)$. Let $ G_\xi $ be the 
grapht we obtain from $ G\setminus 
M_{\alpha_\xi} $ by contracting $ X_\eta $ to $ t_\eta $ for every $ \eta<\xi $. Note that $ G_\xi $ is definable from the 
parameters 
$G,\  \left\langle M_\alpha:\ \alpha \leq\alpha_\xi  \right\rangle $ and $ \prec $. 
\begin{claim}\label{clm: remains Erdos-Menger}
For every  $ \eta<\xi $: $ \delta_{G\setminus M_{\alpha_\xi}}(X_\eta) $ is an Erdős-Menger $ t_\eta $-cut in  $ G\setminus 
M_{\alpha_\xi} $ and hence $ X_{\eta} $ is boundary-linked w.r.t. $ t_\eta $ in $ G\setminus 
M_{\alpha_\xi} $.
\end{claim}
\begin{proof}
 Let $ \eta<\xi $ be given. By condition (\ref{item: path-system}),  $ \delta_{G\setminus M_{\alpha_\eta}}(X_\eta) $ is an 
 Erdős-Menger $ t_\eta $-cut in $ G\setminus M_{\alpha_\eta} $. Let $ \mathcal{P}$ be a path-system that witnesses this. 
 Since $ X_\eta, M_{\alpha_\eta}\in M_{\alpha_\xi} $ we can assume that $ \mathcal{P}\in M_{\alpha_\xi} $.  But then we 
 claim that for  every $ e\in \delta_{G\setminus M_{\alpha_\xi}}(X_\eta) $ the unique $ P_e\in \mathcal{P} $ 
through $ e $ has no edges in $ M_{\alpha_\xi} $. Indeed, if $ f\in E(P_e) \cap M_{\alpha_\xi}$ then $ P_e\in 
M_{\alpha_\xi} 
$ because $ P_e $ is 
definable from $ f, 
\mathcal{P}\in M_{\alpha_\xi} $ but then $ P_e\in M_{\alpha_\xi} $ and hence $ e\in E(P_e)\subseteq M_{\alpha_\xi} $ 
which contradicts $ e\in 
\delta_{G\setminus 
M_{\alpha_\xi}}(X_\eta) $. Thus $ \mathcal{P}\setminus M_{\alpha_\xi} $ witnesses that $ \delta_{G\setminus 
M_{\alpha_\xi}}(X_\eta) $ is an Erdős-Menger $ t_\eta $-cut in  $ G\setminus 
M_{\alpha_\xi} $. 
\end{proof}
\begin{corollary}\label{cor: t xi not linkable in G xi}
 $ t_\xi $ is not linked in $ G_\xi $.
\end{corollary}
\begin{proof}
Suppose for a contradiction that $ \mathcal{P} $ links $ t_\xi $ in $ G_\xi $. Then Claim \ref{clm: 
remains Erdos-Menger} ensures that 
the Lifting lemma (Observation \ref{obs: lift up}) is 
applicable to $ \mathcal{P} $ in $ G\setminus M_\xi $ 
with the sets  
\[ X_{t}:=
\begin{cases} X_\eta &\mbox{if } t=t_\eta\text{ for some } \eta<\xi \\
\{ t \} & \mbox{if }t\in T\setminus \{ t_\eta:\ \eta <\xi \} . 
\end{cases}  \]Since $ \delta_{G_\xi}(t_\xi)=\delta_{G\setminus 
M_{\alpha_\xi}}(t_\xi)=\delta_G(t)$, we conclude that $ t_\xi $ is 
linked in $ G\setminus M_{\alpha_\xi} $ which contradicts the choice of $ t_\xi $.
\end{proof}

We take the $ \trianglelefteq $-largest Erdős-Menger $ t_\xi$-cut in $ G_{\xi} $ and define $ X_\xi $ to be the $ t_\xi $-side 
of it. The preservation of all the properties but (\ref{item: uses an edge down}) 
follows directly from the construction. We show that (\ref{item: uses an edge down}) is also preserved. It follows by the 
Lifting lemma (Observation \ref{obs: lift up}) that $ 
\delta_{G_{\xi}}(X_\xi)( =\delta_{G\setminus M_{\alpha_\xi}}(X_\xi) )$  is an Erdős-Menger $ t_\xi $-cut in $ 
G\setminus M_{\alpha_\xi} $ as 
well. 
Note that $\delta_{G\setminus M_{\alpha_\xi}}(X_\xi)\neq \delta_G(t_\xi) $ because otherwise  $ t_\xi $ would be linked in 
$ 
G\setminus 
M_{\alpha_\xi} $ by any path-system that witnesses that $ \delta_{G\setminus M_{\alpha_\xi}}(X_\xi) $ is an Erdős-Menger 
$ t_\xi $-cut in $ G\setminus M_{\alpha_\xi} $.  Since $ \delta_G(t_\xi) $ is the only Erdős-Menger $ t_\xi$-cut in  $ G $ we 
conclude that $ \delta_G(X_{\xi}) $ is not an Erdős-Menger $ t_\xi $-cut in $ G $. Therefore $\delta_G(X_{\xi})\cap 
M_{\alpha_\xi}\neq 
\emptyset  $ since otherwise any path-system that witnesses that $\delta_G(X_{\xi})\cap M_{\alpha_\xi}\ $ is an 
Erdős-Menger 
$ t_\xi $-cut in $ G\setminus M_{\alpha_\xi} $ also witnesses this in $ G $ contradicting the uniqueness of $ \delta_G(t_\xi) 
$. The recursion is 
done.

Let $ S' $ be the set of limit ordinals in $ S $. Then $ S' $ is still stationary. For $ \alpha_\xi \in S' $ we let $ f(\alpha_\xi) $ be 
the smallest 
ordinal 
$ \beta $ with  
$\delta_G(X_{\xi})\cap M_{\beta}\neq \emptyset  $. This $ f $ is well-defined because $\delta_G(X_{\xi})\cap 
M_{\alpha_\xi}\neq \emptyset  $ by (\ref{item: uses an edge down}). 
Furthermore, $ f(\alpha_\xi)<\alpha_\xi $ because $ M_{\alpha_\xi}=\bigcup_{\beta<\alpha_\xi}M_\beta $ since $ 
\left\langle M_\alpha:\ 
\alpha<\kappa  \right\rangle $ is continuous.  By Fodor's lemma (Lemma \ref{lem: Fodor}) we can find a stationary $ 
S''\subseteq S' $ and a $ \gamma<\kappa $ such that $ f(\alpha)=\gamma $ for every $ \alpha \in S'' $. Then each of the $ 
\kappa $ many 
cuts $ \delta_G(X_\xi)$  $ (\alpha_\xi \in S'') $ meets $ M_\gamma $ but an  $e\in M_\gamma $ can belong to at most 
two of them which is a contradiction because $ \left|M_\gamma\right|<\kappa $.
\end{proof}

\begin{corollary}\label{cor: suitable submodel}
Let $ \kappa $ be an infinite cardinal,  $ \mathcal{X} $ a set of size at most $ \kappa 
$ and let  $ G=(V,E,I,T) $ be a grapht such that for every $ t\in T $, $ \delta(t) $ 
is the unique Erdős-Menger $ t $-cut. Then there exists an elementary submodel $ M $ of size $ \kappa $ with $ 
\mathcal{X}\cup\{ G \}\subseteq M  $ such that 
$ G \setminus M $ 
satisfies the linkability condition.
\end{corollary}
\begin{proof}
We build a sequence as in Lemma \ref{lem: no stat error} and find a club of suitable submodels. In more detail: By 
transfinite recursion, we define an increasing, 
continuous sequence $ \left\langle M_\alpha:\ \alpha<\kappa^{+}  \right\rangle $  of elementary submodels with $ \left\langle 
M_\beta:\ 
\beta\leq\alpha  \right\rangle\in M_{\alpha+1} $, $ \left|M_\alpha\right|=\kappa $ and $\mathcal{X}\cup\{ G \}\subseteq 
M_{\alpha} $ for each $ 
\alpha<\kappa^{+} $.  Let $ M_0 $ be an elementary submodel  of size $ 
\kappa 
$ with $ \mathcal{X}\cup\{ G \}\subseteq M_0  $. Suppose that  $ \alpha< \kappa^{+}$ and  $ M_\beta $ is already defined 
for $ \beta<\alpha $. 
If $ \alpha  $ is a limit ordinal we let $ M_\alpha:=\bigcup_{\beta<\alpha}M_\beta 
$. If $ \alpha=\beta+1 $, then we let $ M_{\beta+1} $ to be an elementary submodel of size $ \kappa $ with $  
\mathcal{X}\cup\{ G, 
\left\langle M_\gamma:\ \gamma\leq \beta  \right\rangle \}\subseteq M_{\beta+1} $. The recursion is done and by Lemma 
\ref{lem: no stat error} we can find not just one but club 
many elementary submodels in the constructed sequence that are as desired.
\end{proof}
\subsection{Induction and singular compactness}
We prove Theorem \ref{thm: linkage} by transfinite induction on  $ \left|E\right|=:\kappa $. We defer the proof of the case $ 
\kappa\leq\omega $ 
and will return to it in Section \ref{sec: countable case}.
Suppose now that $ \kappa>\omega $. By Lemma \ref{lem: joker} (applied with the whole $ T $) we can assume without 
loss of generality that for each $ 
t\in T $, $ \delta(t) $ 
is the unique Erdős-Menger $ t $-cut. Assume first that $ \kappa $ is regular.  Let $ \left\langle 
M_\alpha:\ 
\alpha<\kappa  \right\rangle $ be an increasing, 
continuous sequence of elementary submodels with $ \left\langle M_\beta:\ \beta\leq\alpha  \right\rangle\in M_{\alpha+1} $ 
and $ \left|M_\alpha\right|<\kappa $ for each $ \alpha<\kappa $ and $ G\in M_0 $. Let $ U $ be a club of $ \kappa $ as in 
Lemma \ref{lem: no stat 
error} and let $ \left\langle \alpha_\xi:\ \xi<\kappa  \right\rangle $ be the increasing enumeration of $ U $. By Observation 
\ref{obs: prev inner Euler} and by the choice of $ U $ we know for each $ \xi<\kappa $, that the grapht $ G\setminus 
M_{\alpha_\xi} $ is inner 
Eulerian and satisfies the linkability condition. Since $ G, M_{\alpha_\xi}\in  M_{\alpha_{\xi+1}}$, we have $ G\setminus 
M_{\alpha_\xi}\in 
M_{\alpha_{\xi+1}} $. Thus by applying Observations \ref{obs: prev inner Euler} and \ref{obs: in-prev linkability} we 
conclude that the graphts 
$ G \cap M_{\alpha_0}=:G_{-1} $ and
$ (G\setminus M_{\alpha_\xi})\cap M_{\alpha_{\xi+1}}=:G_\xi $ for $ 0\leq\xi<\kappa $ are inner Eulerian and satisfy the 
linkability condition. 
By induction we can take a perfect linkage $ \mathcal{P}_\xi $ of $ G_\xi $ for $ -1\leq \xi <\kappa $. Since 
the sequence $ \left\langle M_\alpha:\ 
\alpha<\kappa  \right\rangle $ is continuous and $ U $ is a club of $ \kappa $, the subsequence $ \left\langle M_{\alpha_\xi}:\ 
\xi<\kappa  \right\rangle $ is also continuous. Thus the graphts $ G_{\xi} $  form an 
edge-partition of $ G $. But then $ \bigcup_{-1\leq\xi<\kappa}\mathcal{P}_{\xi} $ is a perfect linkage of $ G $.

Now we turn to the case where $ \kappa $ is singular. The proof in this case is based on Shelah's singular compactness 
technique.  Let $ \left\langle \kappa_\alpha:\ \alpha<\mathsf{cf}(\kappa)  \right\rangle $ be an 
increasing, continuous sequence of cardinals with limit $ \kappa $ in which $\kappa_0\geq\mathsf{cf}(\kappa) $. We define a 
``matrix''  
  \[ \{ M_{\alpha, n}:\ \alpha<\mathsf{cf}(\kappa), n<\omega \} \]
  of elementary submodels with the following properties:
  
 \begin{enumerate}
  \item\label{item: sing size} $ \left|M_{\alpha,n}\right|=\kappa_\alpha $,
  \item\label{item: sing contains size} $\kappa_\alpha\subseteq M_{\alpha,n} $,
  \item\label{item: sing contains previous} $ M_{\alpha', n'}\in M_{\alpha,n}$ if either  $ n'<n $ or $ n'=n $ and $ 
 \alpha'<\alpha $.
 \item\label{item: G}$ G\in M_{0,0} $,
 \item\label{item: linkel}$ G\setminus M_{\alpha, n} $ satisfies the linkability condition.
 \end{enumerate}
 \begin{fact}[{\cite[Claim 3.7]{soukup2011elementary}}]\label{fact: elementary contains}
  If $ M $ is an elementary submodel with $ X\in M $ and $ \left|X\right|\subseteq M $, then $ X\subseteq M $.
  \end{fact}
 \begin{corollary}\label{cor: elementary contains}
 We have $ M_{\beta,n}\subseteq M_{\alpha,n} $ and $ M_{\beta,n}\subseteq M_{\beta, n+1} $ whenever $ 
 \beta<\alpha<\mathsf{cf}(\kappa) $ and $ n<\omega $.
 \end{corollary}
 Let $ M_{\alpha}:=\bigcup_{n<\omega}M_{\alpha,n} $ for $ \alpha<\mathsf{cf}(\kappa) $.

\begin{lemma}\label{lem: the sequence is continuous}
$ M_\alpha=\bigcup_{\beta<\alpha}M_\beta $ for every  limit ordinal $ \alpha< \mathsf{cf}(\kappa) $.
\end{lemma}
\begin{proof}
Let a limit ordinal  $ \alpha< \mathsf{cf}(\kappa) $ be fixed. We start with the easy direction: $ M_\alpha\supseteq 
\bigcup_{\beta<\alpha}M_\beta 
$.  Let $ \beta<\alpha $ be given. By 
Observation 
\ref{cor: elementary contains}, 
$ M_{\beta,n}\subseteq M_{\alpha,n} $ for every $ n<\omega $, therefore
\[ M_\beta=\bigcup_{n<\omega}M_{\beta,n} \subseteq \bigcup_{n<\omega}M_{\alpha,n}=M_\alpha. \]

 It remains to prove that $ 
M_\alpha\subseteq \bigcup_{\beta<\alpha}M_\beta $. To do so, it is enough to show that $ M_{\alpha,n}\subseteq 
\bigcup_{\beta<\alpha}M_\beta 
$ for each $ n<\omega $. Let $ n<\omega $ be fixed. By Observation \ref{cor: elementary contains} we have $ 
M_{\alpha,n}\in M_{0,n+1}\subseteq M_0\subseteq 
\bigcup_{\beta<\alpha}M_\beta $. By property (\ref{item: sing size}),  $ \left|M_{\beta,n} \right|=\kappa_\beta $ and hence 
$ \left|M_\beta\right|=\kappa_\beta $  for each $ \beta<\alpha $.
We also know that $ \kappa_\alpha=\bigcup_{\beta<\alpha}\kappa_\beta $ since $ \left\langle \kappa_\beta:\ 
\beta<\mathsf{cf}(\kappa)  \right\rangle $ is 
continuous. Thus 
\[ \left|M_{\alpha,n}\right|=\kappa_\alpha=\bigcup_{\beta<\alpha}\kappa_\beta
\subseteq\bigcup_{\beta<\alpha}M_\beta. \] The set  
$\bigcup_{\beta<\alpha}M_\beta$ is an elementary submodel because it is the union of a chain of elementary submodels. 
Since we have shown $ 
M_{\alpha,n}\in\bigcup_{\beta<\alpha}M_\beta $ and 
$ \left|M_{\alpha,n}\right|\subseteq \bigcup_{\beta<\alpha}M_\beta$  we may conclude $ M_{\alpha,n}\subseteq 
\bigcup_{\beta<\alpha}M_\beta$ by Fact \ref{fact: elementary contains}.
\end{proof}

\begin{lemma}\label{lem: perfect linkage in G alpha n}
For every $ \alpha<\mathsf{cf}(\kappa) $ and $ n<\omega $, the grapht $G_{\alpha,n}:= G\cap(M_{\alpha,n}\setminus 
M_{\alpha,n-1}) $ has a 
perfect linkage $ \mathcal{P}_{\alpha,n} $  (where $ 
M_{\alpha,-1}:=\emptyset $).
\end{lemma} 
\begin{proof}
The premise of Theorem \ref{thm: linkage} is preserved for $ G\setminus M_{\alpha,n-1}\in M_{\alpha,n}$ by property 
(\ref{item: linkel}) and  Observation \ref{obs: prev inner Euler}. Thus it is preserved for $ 
(G\setminus M_{\alpha,n-1})\cap M_{\alpha,n}= G_{\alpha,n} $ as well because of Observation 
\ref{obs: in-prev linkability}. Since $ 
\left|E(G_{\alpha,n})\right|=\kappa_\alpha<\kappa $,  we can apply induction to obtain a perfect linkage in $ G_{\alpha,n} $. 
\end{proof}

Then $ \mathcal{P}_\alpha:=\bigcup_{n<\omega}\mathcal{P}_{\alpha,n} $ is a perfect linkage in $ G\cap M_\alpha $.
By property (\ref{item: sing contains previous}) and Fact \ref{fact: elementary contains} we know that  \[ 
M_{\alpha+1,n}, M_{\alpha+1,n-1}\in 
M_{\alpha, n+1}\subseteq M_\alpha. \]   Therefore $ 
G_{\alpha+1,n}\in 
M_{\alpha} $. Hence by Lemma \ref{lem: perfect linkage in G alpha n} we can pick a perfect linkage $ 
\mathcal{P}_{\alpha+1,n}\in M_\alpha $ for $ G_{\alpha+1,n} $. But then 
 we claim that $ \mathcal{Q}_{\alpha,n}:=\mathcal{P}_{\alpha+1,n}\setminus M_\alpha $  is a perfect linkage in $ 
 G_{\alpha+1,n}\setminus 
 M_\alpha=G\cap(M_{\alpha+1,n}\setminus 
(M_{\alpha+1,n-1}\cup M_\alpha)) $. Indeed, if $ e\in E(G_{\alpha+1,n}) \setminus M_\alpha $ is incident with a terminal 
vertex, then the unique 
$ P_e\in \mathcal{P}_{\alpha+1,n} $ through $ e $ is not in $ M_\alpha $ because $ P_e $ is definable from $ e, 
\mathcal{P}_{\alpha+1,n} \in 
M_\alpha $, furthermore, $ E(P)\cap M_\alpha=\emptyset $ for similar reason. It follows, that $ 
\mathcal{Q}_\alpha:=\bigcup_{n<\omega}\mathcal{Q}_{\alpha,n} $ is a perfect linkage in 
$ G \cap(M_{\alpha+1}\setminus M_\alpha) $. Since Lemma \ref{lem: the sequence is continuous} ensures that $ G\cap 
M_0 $ together with $ G \cap(M_{\alpha+1}\setminus M_\alpha) $  for $ \alpha<\mathsf{cf}(\kappa) $ forms an 
edge-partition of $ G $, the path-system $ \mathcal{P}_0 \cup \bigcup_{\alpha<\mathsf{cf}(\kappa)}\mathcal{Q}_\alpha $ is 
a perfect linkage in $ G $.
\section{Proof of the countable case}\label{sec: countable case}
It this section we prove the restriction of the Linkage theorem  (Theorem \ref{thm: linkage}) to graphts with only  countable 
many edges. First, we need some preparation.

\begin{observation}\label{obs: inner Euler preserved}
If $ G=(V,E,I, T) $ is an inner Eulerian grapht and $ H $ is either a cycle or a $ T $-path in $ G $ then the grapht $ G-E(H) $ 
is inner Eulerian.
\end{observation}
\begin{proof}
For a set $ X\subseteq V \setminus T$ it is easy to see that $ \left| \delta_G(X)\cap E(H)\right| 
$ must be an even natural number. But then if $  d_{G-E(H)}(X) $ is finite, then  $ 
d_{G-E(H)}(X) =d_{G}(X) -\left| \delta_G(X)\cap E(H)\right| $ holds, where both quantities on the 
right side are even and therefore  $ d_{G-E(H)}(X) $ is also even.
\end{proof}
\begin{lemma}\label{lem: rest cycle coverable}
Let  $ G=(V,E,I,\{ s,t \}) $ be an inner Eulerian grapht  with $ \left|E\right|\leq \omega $ where $ s $ is 
linked. Then there is a 
linkage $ \mathcal{P} $ for $ s $ such that there is a family of pairwise edge-disjoint cycles in $ G-E(\mathcal{P}) $ 
covering $ 
\delta_{G-E(\mathcal{P}) }(t) $. 
\end{lemma}
\begin{proof}
It is enough to show that for every $ e\in \delta(s)\cup \delta(t) $ there is a subgraph $ H $ containing $ e $, that is either an $ 
st $-path or a 
cycle through $ t $  that does not contain $ s $,  such that $ G-E(H) $ preserves the premise of Lemma 
\ref{lem: rest cycle coverable}. 

Being inner Eulerian is preserved 
``automatically'' by Observation \ref{obs: inner Euler preserved} thus we only need to maintain the linkability of $ s $. Let $ 
\mathcal{P} $ be any linkage of $ s $. If $ e\in E(\mathcal{P}) $, then the unique $ P_e\in \mathcal{P} $ through $ e $ is 
suitable as $ H $. Suppose that $ e\notin  E(\mathcal{P}) $. Then $ e\notin \delta(s) $ (because $ \mathcal{P} $ links $ s $) 
and hence we must have
$ e\in \delta(t) $. Let $ \mathcal{A} $ be an edge-partition of $ G $ into cycles and $ st $-paths (exists by Corollary 
\ref{cor: inner Eulerian grapht partionon}) and let $ H_e\in \mathcal{A}$ be the unique element through $ e $. If $ 
E(H_e)\cap E(\mathcal{P})=\emptyset $, then $ H:=H_e $ is suitable. 
Otherwise consider the orientation of $ H_e $ where it is a directed $ ts $-path or a directed cycle in which $ e $ is an 
outgoing edge of $ t $ . Let $ v $ be the first vertex of $ V(\mathcal{P}) $ that we reach by going 
along this 
orientation 
starting from $ t $ (see Figure \ref{fig: cycle}) and let $ Q $ be the corresponding segment of $ H_e $ between $ t $ and $ v 
$. Take a path $ P\in 
\mathcal{P} $ witnessing $ v\in V(\mathcal{P}) $. Then the path $ H:=sPvQ $ is suitable because it does not share any 
edges 
with the paths in $ \mathcal{P}-P $ and therefore the linkability of $ s $ in $ G-E(H) $ is witnessed by $ \mathcal{P}-P 
$.

\begin{figure}[h]
\centering

\begin{tikzpicture}

\node (v1) at (-2.5,0) {};
\node (v2) at (2.5,0) {};

\draw  plot[smooth, tension=.7] coordinates {(v1) (-1.1,0.7) (0.7,0.9) (1.9,0.4) (v2)};
\draw  plot[smooth, tension=.7] coordinates {(-2.5,0) (-1.3,-0.1) (-0.1,0.1) (0.7,-0.3) (2,-0.3) (2.5,0)};
\draw  plot[smooth, tension=.7] coordinates {(-2.5,0) (-2.2,-1) (-1.6,-0.9) (-1.3,-0.1) (-0.6,0.4) (-0.1,0.1) (0,-1.6) (1.4,-1.4) 
(2.5,0)};

%cycle
\draw[dashed]  plot[smooth, tension=.7] coordinates {(2.5,0) (0.8,0.1) (0.7,0.9)};

%names
\node at (-2.5,0.3) {$s$};
\node at (2.7,0.2) {$t$};
\node at (0.7,1.1) {$v$};
\node at (1.2,0.3) {$Q$};
\node at (-0.8,1) {$P$};
\node at (2,0.1) {$e$};
\end{tikzpicture}
\caption{The construction of $ H $.} \label{fig: cycle}
\end{figure}
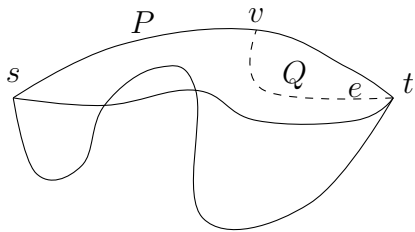
\end{proof}
\begin{lemma}\label{lem: deletion of two edges}
Let  $ G=(V,E,I,\{ s,t \}) $ be an inner Eulerian grapht with  $ \left|E\right|\leq \omega $ where  $ \delta(s) $ is 
the unique 
Erdős-Menger $ st $-cut. Then after the deletion of at most two edges, $ s $ remains linked.
\end{lemma}
\begin{proof}
Suppose for a contradiction that there are (not necessarily distinct) $ f, g\in E $ such that $ s $ is not linked in $ G-f-g $. We 
may assume that 
$ \delta_G(s)\cap \{ f,g 
\}=\emptyset $ since otherwise the deletion of the edge set $ \{ f,g \}\setminus \delta_G(s) $ already ruins the linkability of $ 
s $. We know that $ s $ must remain linked in $ G-f $ since otherwise by applying 
Lemma \ref{lem: tight cut}
with $ G $ and $ f $ we obtain an Erdős-Menger $ st $-cut that contains $ f $ and hence distinct from 
$ \delta_{G}(s) $ contradicting the assumption that $ \delta_{G}(s) $ is the unique Erdős-Menger $ st $-cut.  

By the indirect assumption, $ s $ is not linked in $ G-f-g $ and we have just seen that it is in $ G-f $. Thus Lemma 
\ref{lem: tight cut} can be applied with $ G-f 
$ and $ g$. 
This provides a tight Erdős-Menger $ st $-cut $ C $  in $ G-f $ that contains $ g $. This $ C $ cannot be an Erdős-Menger 
$ st $-cut in $ G $ because $ \delta(s) $ is the unique one and $ g\in C\setminus \delta(s) $. Therefore $ f $ connects the two 
sides of $ C $, i.e. $ C+f $ is an $ st $-cut in $ G $ (but not an Erdős-Menger 
$ st $-cut because $ f\in (C+f)\setminus \delta(s)$).  Contract the $ t $-side of $ C+f $  to $ t $ to obtain $ G' $.  Note that $ t 
$ is not linked in $ G' $.
The 
tightness 
of $ C $ in $ G-f $ ensures 
that  $ G' $ with $ s $ and $ t $ satisfies the premise of Lemma \ref{lem: missing exactly one}. Therefore every linkage of
$ s $ in $ G' $ misses exactly one edge from $ \delta_{G'}(t) $. 

On the other hand, by Lemma \ref{lem: rest cycle coverable} there exists a linkage $ \mathcal{P} $ for $ s $ in $ 
G' $ such that there is a family of pairwise 
edge-disjoint cycles in $ G'-E(\mathcal{P}) $ covering $ \delta_{G'-E(\mathcal{P}) }(t) $. Since a cycle through $ t $ uses an 
even number of edges 
incident with $ t $, the set $ \delta_{G'-E(\mathcal{P}) }(t) $ cannot be a singleton. But then the existence of this $ 
\mathcal{P} $ contradicts the last sentence of 
the previous paragraph.
\end{proof}

\begin{corollary}\label{cor: deletion of two edges}
Let  $ G=(V,E,I,T) $ be an inner Eulerian grapht with $ \left|E\right|\leq \omega $. Then after the deletion of at most two 
edges all those $ t\in T $ 
remains linked for which $ \delta(t) $ was the unique Erdős-Menger $ t $-cut. 
\end{corollary}
\begin{proof}
For a given $ t\in T $ we contract $ T-t $ to a $ t' \in T-t $ and apply Lemma \ref{lem: deletion of two edges}.
\end{proof}

The restriction of the Linkage theorem (Theorem \ref{thm: linkage}) to graphts with countably many edges follows from the 
following lemma by a straightforward recursion:
\begin{lemma}\label{lem: 1path make}
Let  $ G=(V,E,I,T) $ be an inner Eulerian grapht with $ \left|E\right|\leq \omega $ that satisfies the linkability condition. Then 
for every  
 $e\in \bigcup_{t\in T}\delta(t)$ there exists a $T$-path $P$ through $e$ such that the grapht $G-E(P)$ is inner Eulerian and 
 satisfies the 
linkability condition.
\end{lemma}

\begin{proof}
By Observation \ref{obs: inner Euler preserved}, $ G-E(P) $ is inner Eulerian for every $ T $-path $ P $, thus it is enough to 
focus on the preservation of the linkability 
condition. Suppose for a contradiction that the lemma is false. Observe that deletion of edges with both end-vertex in $ T $ 
does not ruin the linkability condition. Thus if $ G=(V,E,I,T) $ and $e_0\in \bigcup_{t\in T}\delta(t)$ 
form a counterexample, then  $ e_0$ has only one of its endpoints, $ s_{e_0}$ say, in $ T $. We associate the following 
quantity with counterexamples:

\[ \left|\left|(G, e_0)\right|\right|:=\min \{ \left|E(P_{e_0})\right|:\ \mathcal{P}\text{ is a linkage of }s_{e_0},\ e_0\in 
E(P_{e_0}),\ 
P_{e_0}\in 
\mathcal{P} \}. \]
Let us fix a counterexample $ (G,e_0) $ that minimizes $ \left|\left|(G, e_0)\right|\right|  $ and write simply $ s $ for $ 
s_{e_0} $.  
 Let $\mathcal{P}=\{ P_e:\ e\in \delta_G(s) \} $ be a path-system where the minimum is taken where $ e\in E(P_e) $ for $ 
 e\in \delta_G(s) $.
\begin{claim}\label{clm: spec counterexample}
We can assume without loss of generality that for every $ t\in T-s $: $ \delta_G(t) $ is the unique Erdős-Menger $ t $-cut in $ 
G $. 
\end{claim}
\begin{proof}
We show that our original counterexample can be modified in order to satisfy this additional property. Let $ \mathcal{F}=\{ 
X_t:\ t\in T \} $ be a family that we obtain by applying Lemma \ref{lem: joker} with $ T'=T-s $. Then \ref{item: joker 
EM}  of Lemma \ref{lem: joker} ensures that $ \delta_{G/\mathcal{F}}(t) $ is the unique Erdős-Menger $ t $-cut in $ 
G/\mathcal{F}=:G' $ for $ t\in T' $.  
For $ 
e\in \delta_G(s)=\delta_{G'}(s)$, let  $ 
P'_e $ be the $ T $-path in  $ G' $ defined by the initial segment of $ P_e $ from $ s $ until its first common 
vertex with $ 
\bigcup_{t\in T'}X_t $ and let $ \mathcal{P}':=\{ P'_e:\ e\in \delta_{G'}(s) \} $. Obviously  $\left|E(P'_{e_{0}})\right|\leq 
\left|E(P_{e_{0}})\right|$, thus if $ G' 
$ and $ e_0 $ form a counterexample, then it is also minimizing in the sense that $\left|\left|(G',e_0)\right|\right|= 
\left|\left|(G,e_0)\right|\right| $.

 Suppose for a contradiction that $ (G',e_0) $ is not a counterexample and let $ Q_{e_0} $ be a $ T $-path through $ e_0 $ 
 such that $ G'-E(Q_{e_0}) $ satisfies the 
 linkability condition.  For $ t\in T' $  let $ \mathcal{P}_t=\{ P_{t,e}:\ e\in 
 \delta_{G}(X_t) \} $ as in \ref{item: joker links} of Lemma \ref{lem: joker}.  Denote $ t_0 $  
the end-vertex of $ Q_{e_0} $ distinct from $ s $ and  let $ f_0 $ 
be the unique 
edge of $ Q_{e_0} $ incident with $ t_0 $ (see Figure \ref{fig: one more path}). Then $ R_{e_0}:=sQ_{e_0}f_0 P_{t_0, 
f_0}t_0 $ is a $ T $-path in 
$ G $ 
covering $ 
e_0 $.  Note that 

\[ (G-E(R_{e_0}))/\mathcal{F}=G'-E(Q_{e_0}) .\]

Let $ \mathcal{P}'_s $ consist of  the single-edge paths corresponding to $ \delta_G(s) $,  $ 
\mathcal{P}_{t_0}':=\mathcal{P}_{t_0}-P_{t_0, f_0} $ and $  
\mathcal{P}_t':= 
\mathcal{P}_t $ for $ t\in T'-t_0 $. These path-systems are as in the Lifting lemma 
(Observation \ref{obs: lift up})   for $ \mathcal{F} $ in $ G-E(R_{e_0}) $ including that $ \mathcal{P}'_t $ covers $ 
\delta_{G-E(R_{e_0})}(t) $ for $ t\in T $. But then any linkage of a 
$ t\in T $ in $G'-E(Q_{e_0}) $ can be lifted up to a linkage of $ t $ in $ G-E(R_{e_0}) $. 
Therefore $ G-E(R_{e_0}) $ satisfies the linkability condition. This means that $ R_{e_0} $ 
witnesses that $ (G,e_0) $ is not a counterexample, a contradiction. 

\begin{figure}[h]
\centering

\begin{tikzpicture}

\draw  (0.9,0.8) node (v1) {} ellipse (1.4 and 1.4);
\draw  (2.5,-3.2) node (v2) {} ellipse (1.4 and 1.5);

%Boudary

\draw[dashed]  plot[smooth, tension=.7] coordinates {(v1) (0.5,1.2) (0.1,1.2) (-0.8,1.2)};
\draw[dashed]  plot[smooth, tension=.7] coordinates {(v1) (0.2,0.5) (-1,0.3)};
\draw[dashed]  plot[smooth, tension=.7] coordinates {(v1) (0.9,0.1) (0.4,-0.1) (-0.4,-0.3)};
\draw[dashed] plot[smooth, tension=.7] coordinates {(v2) (1.8,-2.9) (0.1,-2.7)};
\draw[dashed] plot[smooth, tension=.7] coordinates {(v2) (2,-3.7) (0.4,-4.1)};
\draw[dashed]  plot[smooth, tension=.7] coordinates {(v2) (2.5,-4) (2,-4.2) (1,-4.6)};
\draw[dashed]  plot[smooth, tension=.7] coordinates {(v2) (3,-2.6) (3,-2.1) (2.9,-1.5)};

%Paths from s

\draw  plot[smooth, tension=.7] coordinates {(-2.2,-1.9) (-2.1,-1.1) (-2.5,-0.6) (-2.1,0.2) (-1,0.3)};
\draw  plot[smooth, tension=.7] coordinates {(-2.2,-1.9) (-1.4,-1.5) (-1.2,-0.8) (-0.4,-0.3)};
\draw  plot[smooth, tension=.7] coordinates {(-2.2,-1.9) (-1.5,-2.7) (0.1,-2.7)};

%names
\node at (-2.4,-2.1) {$s$};
\node at (-2,-2.5) {$e_0$};
\node at (-0.2,-2.4) {$Q_{e_0}$};
\node at (2.7,-3.4) {$t_0$};
\node at (1,-2.6) {$f_0$};
\node at (1.1,0.9) {$t_1$};
\node at (0,0.8) {$\mathcal{P}_{t_1}$};
\node at (0.7,-3.2) {$R_{e_0}$};
\node at (2,-2.7) {$P_{t_0,f_0}$};
\draw  plot[smooth, tension=.7] coordinates {(-1,0.3) (0.2,0.5)};
\draw  plot[smooth, tension=.7] coordinates {(-0.4,-0.3) (0.4,-0.1)};
\draw  plot[smooth, tension=.7] coordinates {(0.1,-2.7) (1.8,-2.9)};
\end{tikzpicture}
\caption{The construction of $ R_{e_0} $.} \label{fig: one more path}
\end{figure}
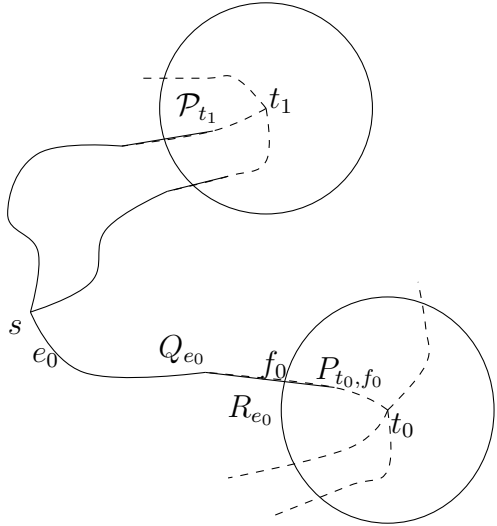
\end{proof}

Since only one endpoint of $ e_0 $ is in $ T $, we  have
$\left|E(P_{e_{0}})\right| \geq 2$. Let $f_0\in E(P_{e_0})$ be the edge next to $e_0$ in $ P_{e_0} $. We obtain $ G' $ from 
$ G 
$ by deleting the edges $e_0$ 
and $f_0$ and adding a single new edge $h_0$ connecting $s$ and the end-vertex of $f_0$ that is not shared with $e_0$ (i.e. 
we
split off $ e_0 $ and $ f_0 $, a technique by Lovász 
from \cite{lovasz1976some}). Let 
$P_{h_0}$ be the path in $ G' $ with $E(P_{h_0})=E(P_{e_0})-e_0-f_0+h_0$ and let us define 
$\mathcal{P}':=\mathcal{P}-P_{e_0}+P_{h_0}$.
 For $ X\subseteq V\setminus T $, the quantities $ d_{G}(X) $ and $ d_{G'}(X) $ are either both infinite or they have the 
 same 
 parity, 
 thus $ G' $ is also inner Eulerian. The linkability condition for $ s $  in $ G' $ is witnessed by $ \mathcal{P}' 
 $. 
 Let $ 
 t\in T-s $ be arbitrary. The linkability condition for $ t $ holds  in $ G-e_0-f_0 $ by 
 Corollary \ref{cor: deletion of two edges} via Claim \ref{clm: spec counterexample} and hence in $ G' $ as well  (if $ h_0\in 
 \delta_{G'}(t)  $, then $ h_0 $ is an edge between $ s $ and $ t $  and hence 
 a $ T 
 $-path itself). 
Note that $G'$ and $h_0$ cannot form a counterexample for Lemma \ref{lem: 1path make}
because $ \left|\left|(G',h_0)\right|\right| =\left|\left|(G,e_0)\right|\right|-1 $ by 
$\left|E(P_{h_{0}})\right|=\left|E(P_{e_{0}})\right|-1 
$. Therefore we can pick some $ T $-path $ P' $ in $ G' $ 
through 
$ h_0 $ such that the linkability condition holds in $ G'-E(P') $. Let us take then a $ T $-path $ P$ in $ G$ through 
$ e_0 $ with $ E(P)\subseteq E(P')-h_0+e_0+f_0 $. Since  $ G'-E(P') $ is a subgraph of $ G-E(P) $ and $ 
\delta_{G-E(P)}(t)=\delta_{G'-E(P')}(t) $ holds for $ t\in T $, the linkability condition in $ G'-E(P') $ implies the linkability 
condition
in $ G-E(P) $. This contradicts the fact that $ G $ and $ e_0 $ form a counterexample to Lemma \ref{lem: 1path make}.
\end{proof}

\section{Outlook}\label{sec: outlook}
The minimax theorems due to Gallai and Mader (independently)  about packing edge-disjoint, vertex-disjoint and internally 
vertex-disjoint $ T $-paths in arbitrary finite graphs \cite{mader1978maximalzahl, gallai1964maximum, 
mader1978maximalzahlH} provide a promising material for structural infinite generalizations.  Our 
corresponding 
conjectures have already been discussed in the 
last section of \cite{joo2023lovcher}. 

Let us formulate the following (somewhat vague) meta-conjecture:
\begin{conjecture}\label{conj: meta}
Theorems that can be proved by augmenting paths for finite structures remain true structurally (possibly under some 
restrictions) for infinite ones.
\end{conjecture}

The infinite generalization of: Menger's theorem \cite{aharoni2009menger}, Ford-Fulkerson theorem \cite{aharoni2011max}, 
Tutte's 
matching theorem \cite{aharoni1988matchings} (together with the Gallai-Edmonds structure theorem \cite[Theorem 
3.1]{aharoni1990lp}), Edmonds' matroid intersection theorem 
\cite{joo2021MIC} \footnote{The 
maxflow-mincut theorem of Ford and Fulkerson and Edmonds' matroid intersection theorem are extended only to countable 
structures for now.} and the Lovász-Cherkassky theorem (Theorem \ref{thm: LCh inf}) are 
special instances of Conjecture \ref{conj: meta}
and they do not require any restrictions. Dilworth' theorem \cite{dilworth1987decomposition} states that every finite poset 
admits 
a partition $ \mathcal{P} $ 
into 
chains such that there is an antichain $ A $ with $ A \cap C\neq \emptyset $ for every $ C\in \mathcal{P} $. It is not too 
hard to construct an infinite poset that cannot be partitioned into finitely many chains (because admitting arbitrary large finite 
antichains) but does not even contain an 
infinite 
antichain. Indeed,   \[ P:=\{ (m,n)\in \mathbb{N} \times \mathbb{N}:\ n \leq m \} \]  where $ (m,n)\leq (m',n') $ iff $ m\leq m' 
$ has this property. However, Dilworth' theorem holds for posets without infinite chains 
\cite[Theorem 7.1]{aharoni1991infinite} (for countable posets, excluding 
only more 
specific substructures is already sufficient \cite{aharoni2011strongly}). Since Dilworth' theorem admits an augmenting path 
based proof,  it fits under the umbrella of 
Conjecture \ref{conj: meta} with the restriction that we forbid infinite chains.  Based 
on a theorem of Gallai and Milgram \cite{gallai1960verallgemeinerung}, we conjecture the following 
generalization:

\begin{conjecture}[infinite Gallai-Milgram]\label{conj: gallaimilgram}
Every digraph $ D $ without infinite directed paths admits a vertex-partition $ \mathcal{P} $ into directed paths and an 
independent 
set $ U $ of vertices such that $ U $ meets every $ P\in \mathcal{P} $.
\end{conjecture}
\noindent The restriction of Conjecture \ref{conj: gallaimilgram}
to acyclic digraphs follows from the infinite version of König's theorem \cite{aharoni1984konig} the same way the infinite 
Dilworth' theorem  does\footnote{This was observed by Aharoni 
(personal communication).} (see the proof of \cite[Theorem 7.1]{aharoni1991infinite}). 

Introducing new specific conjectures based on Conjecture \ref{conj: meta} is quite easy but proving them is usually a highly 
non-trivial task. Let us mention a second example based on \cite[Theorem 2]{frank1984covering}:

\begin{conjecture}[infinite Frank-Sebő-Tardos]\label{conj: frank tardos sebo}
Let $ G=(S,T, E) $ be a (possibly infinite) $ 2 $-edge-connected bipartite graph. Then there is a strong orientation $ D $ of $ 
G $ and a partition $ \mathcal{P} $ of $T $ such that for every $ U\in \mathcal{P} $, each weak component of $ D-U $ 
sends 
only one directed edge toward $ U $.
\end{conjecture}

\printbibliography
\end{document}